\documentclass[12pt]{article}
\usepackage{geometry}                
\usepackage{graphicx}
\usepackage{amssymb}
\usepackage{amsmath}
\usepackage{amsthm}
\usepackage{verbatim} 
\usepackage{enumitem} 
\usepackage{centernot} 
\usepackage{hyperref} 
\usepackage{overpic} 

\usepackage[usenames,dvipsnames]{xcolor}

\DeclareMathOperator{\arccosh}{arccosh}
\DeclareMathOperator{\sff}{II} 
\DeclareMathOperator{\spn}{span}

\theoremstyle{plain}
\newtheorem{thm}{Theorem}
\newtheorem{lem}{Lemma}
\newtheorem{prop}{Proposition}
\newtheorem{cor}{Corollary}

\theoremstyle{definition}
\newtheorem{defn}{Definition}

\theoremstyle{remark}
\newtheorem*{rmk}{Remark}

\begin{document}
\title{Hilbert's Theorem, via moving frames}
\author{William D. Dunbar}
\date{\today}
\maketitle

\begin{abstract}
We present a proof that the hyperbolic plane cannot be isometrically immersed in Euclidean $3$-space by a $C^\infty$ map.  Ideas from many topics in (essentially) undergraduate mathematics are applied;
the use of moving frames and connection forms to express the geometry simplifies the outline of the proof, compared to, say, using coordinate patches and Christoffel symbols.  The key transition is from expressions in terms of the principal directions on the immersed surface (which give access to the Gaussian curvature) to expressions in terms of the asymptotic directions (which yield a coordinate system and give access to surface area).
\end{abstract}

\section{Outline of the proof} \label{Sec:bullets}

\begin{itemize} 
\item Assume for simplicity that $\phi: \mathbb{H}^2 \rightarrow \mathbb{R}^3$ is a $C^\infty$ isometric embedding (immersions are no more difficult).  In this section, ``the surface'' will refer to the image of $\phi$.  
\item Construct an orthonormal frame field $\{e_1,e_2\}$ on $\mathbb{H}^2$ which is mapped by $\phi_*$ to unit vectors in the principal directions on the surface (maximizing and minimizing normal curvature).
\item Construct a frame field $\{ E_1,E_2 \}$ on $\mathbb{H}^2$ which is mapped by $\phi_*$ to linearly independent unit vectors in asymptotic directions (where the normal curvature equals zero).
\item Express the connection $1$-form $\omega_1^2$ (for $\{e_1,e_2\}$) first in terms of the $1$-forms dual to $\{e_1,e_2\}$ and the principal curvatures [Lemma \ref{L:connprinequ}], then in terms of the same forms and the angle $\alpha$ between either of the asymptotic directions and $e_1$ [Lemma \ref{L:connalpha}], and finally change basis to the $1$-forms dual to $\{E_1,E_2\}$  [equation \eqref{Equ:conneta}].  
(The Gaussian curvature on $\mathbb{H}^2$ being identically equal to $-1$ is used in the proof of Lemma \ref{L:connalpha}.)
\item Show that the dual $1$-forms $\{\eta^1, \eta^2\}$ for $\{E_1,E_2\}$ are closed forms, hence exact [Lemma \ref{L:closed}]; furthermore, construct a (global!) coordinate chart mapping $P \in \mathbb{H}^2$ to
$F(P) = (u^1(P),u^2(P)) \in \{ (x^1,x^2) \} = \mathbb{R}^2$, such that $\frac{\partial}{\partial u^1} = E_1$ and $\frac{\partial}{\partial u^2} = E_2$ [Proposition \ref{P:coords}].
\item Define $\Theta$ to be the (oriented) angle between $E_1$ and $E_2$; hence,  
$\sin\Theta$ is the area distortion factor when mapping from the $x^1x^2$-plane to $\mathbb{H}^2$.
Show that 
$\sin\Theta = \frac{\partial^2 \Theta}{\partial u^1 \partial u^2}$
[equation \eqref{Equ:KeyPDEinD}].
\item Calculate the area in $\mathbb{H}^2$ corresponding to $[-a,a]\times [-a,a]$ in the $x^1x^2$-plane, and show that as $a \to \infty$, area is bounded above by $2\pi$ [equation \eqref{Equ:finitearea}] \dots
but also calculate the area of $\mathbb{H}^2$ by other means to show that it is infinite [equation \eqref{Equ:infinitearea}].  This is a contradiction, so $\phi$ cannot exist.
\end{itemize}

This approach to Hilbert's theorem was inspired by the argument given by Rubens Le\~{a}o in the Appendix of
\cite{doC1983}.
Citations will be supplied for equations that come from undergraduate differential geometry and/or from tensor analysis, but proofs will usually be omitted.  For differential geometry, 
\cite{ON2006} and
\cite[Chapter 3, Section 3]{Sh2018} use moving frames; 
\cite{MP1977} and
\cite{doC2016}
may also be helpful.
Basic facts about differential forms in $\mathbb{R}^n$ (exterior differentiation, wedge product, pullback) can be found in
\cite[Chapter 1]{doC1994}
and in
\cite[Chapter 7]{SI1979}.

\section{Proof of the main theorem} \label{Sec:prf}

\begin{thm} \label{T:Hilbert}
There does not exist a $C^\infty$ isometric immersion of the hyperbolic plane into Euclidean $3$-space.
\end{thm}

Before starting the proof, we make a foundational definition and a remark on notation.

\begin{defn}
A \emph{moving frame} on an open subset $U$ of $\mathbb{R}^n$ is an ordered $n$-tuple of smooth (i.e., $C^\infty$) vector fields 
$\{v_1,\dots,v_n\}$, such that the vectors at each point $p \in U$ form an orthonormal basis for the tangent space $T_p(U)$.
\end{defn}

The previous definition agrees with
\cite[page 118]{doC1983}; see also
\cite[page 77]{doC1994} and
\cite[Chapter II, Section 1]{ON2006}.
Elsewhere, such as in 
\cite[page 285]{SII1979}, the vectors in a moving frame need only be linearly independent (and span). 

\begin{rmk}
We will use lower subscripts for contravariant objects such as vectors, and upper subscripts for their components with respect to a basis (as in the preceding definition).  Similarly, we will use upper subscripts for covariant objects such as differential forms, and lower subscripts for their components with respect to a basis.  Summation signs are needed only a few times, and they are \emph{not} omitted, as they would be in ``Einstein summation convention''
\cite[pages 50--51, pages 155--158]{SI1979}.
\end{rmk}
%
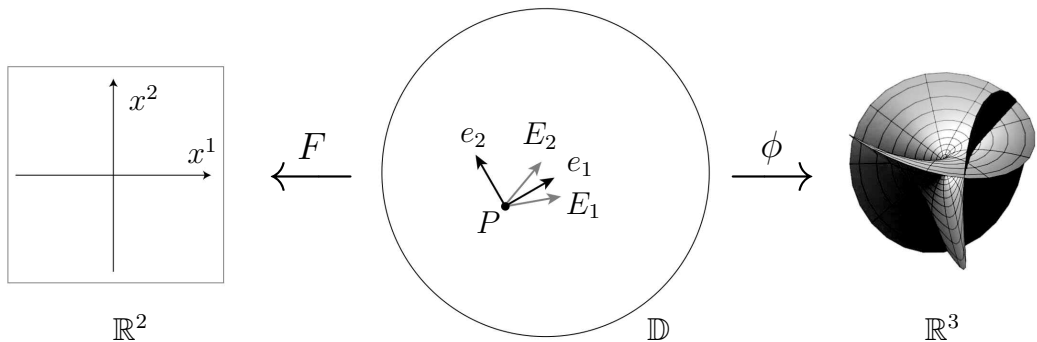
\begin{figure} 
\begin{overpic}[scale=0.36]{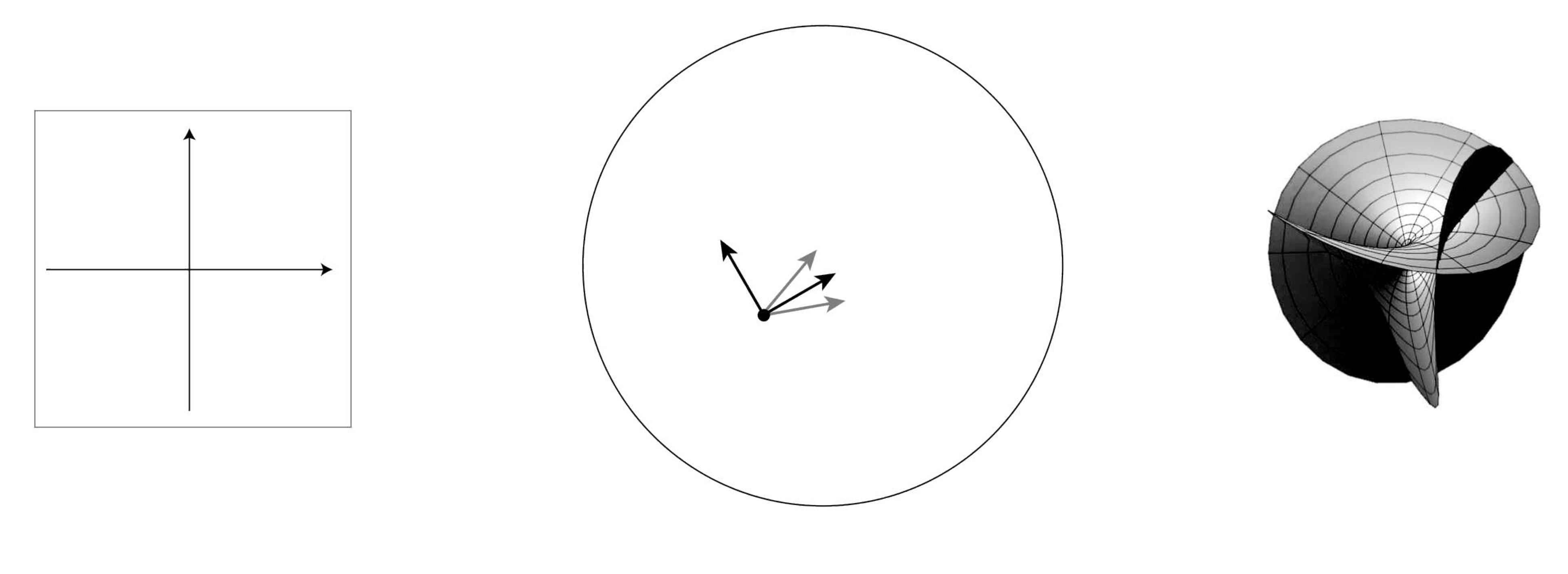} 
\put(12,4){$\mathbb{R}^2$}
\put(62,4){$\mathbb{D}$}
\put(88,4){$\mathbb{R}^3$}
\put(26.5,18){\LARGE$\overset{F}{\longleftarrow}$}
\put(69.5,18){\LARGE$\overset{\phi}{\longrightarrow}$}
\put(19,20.5){$x^1$}
\put(13.5,25.5){$x^2$}
\put(46,14){$P$}
\put(54.5,19.8){$e_1$}
\put(44.5,22.5){$e_2$}
\put(54.4,15.8){$E_1$}
\put(50.3,22.3){$E_2$}
\end{overpic}
\caption{Overview 
} \label{F: mapdiagram}
\end{figure}
\begin{proof}[Proof of Theorem \ref{T:Hilbert}] 
The Poincar\'e disk $\mathbb{D}$ will represent the hyperbolic plane; define $O :=(0,0) \in \mathbb{D}$ (see the Appendix for the Riemannian metric and other details).  
Suppose that $\phi: \mathbb{D} \longrightarrow \mathbb{R}^3$ is a $C^\infty$ isometric immersion.

\subsection{Construct vector fields and $1$-forms on $\mathbb{D}$} \label{SS:setup} 

For all $P \in \mathbb{D}$, there is a positive number $\epsilon_P$, such that $\phi$ is an embedding, when restricted to the open disk $U_P$ of radius $\epsilon_P$ in the hyperbolic metric.
The tangent space to the surface $\phi(U_P) =: \bar{U}_P$ at $\phi(P) =: \bar{P}$ can inherit an orientation from $T_P(\mathbb{D})$, via $(\phi_*)_P$, so there is a preferred unit normal vector at $\bar{P}$, which can be continuously (indeed, smoothly) extended to a unit normal vector field
$\bar{e}_3: \bar{U}_P \rightarrow S^2$.

At every point $Q \in U_P$, the map $(-d\bar{e}_3)_{\bar{Q}} : T_{\bar{Q}}(\bar{U}_P) \rightarrow T_{\bar{e}_3(\bar{Q})}S^2$, is a self-adjoint linear transformation of 
$T_{\bar{Q}}(\bar{U}_P) = \phi_*(T_Q(\mathbb{D}))$
\cite[page 142]{doC2016};
note that
$T_{\bar{e}_3(\bar{Q})}S^2$ and $T_{\bar{Q}}(\bar{U}_P)$ are both names for the orthogonal
complement in $\mathbb{R}^3$ of $\spn(\bar{e}_3(\bar{Q}))$.
The eigenvalues are equal to the principal curvatures (let $\bar{\kappa}_1(\bar{Q})$ denote the positive eigenvalue,  and $\bar{\kappa}_2(\bar{Q})$ the negative eigenvalue), and the corresponding (one-dimensional, orthogonal) eigenspaces 
$\bar{L}_1(\bar{Q}), \bar{L}_2(\bar{Q})$ give the principal directions  
\cite[page 146]{doC2016}.
At this stage, we are using the fact that the Gaussian curvature (product of the principal curvatures) is negative at every point of $\bar{U}_P$, but \emph{not} (yet) the fact that it is constant and equal to $-1$.

Via $(\phi |_{U_P})_*^{-1}$, we can transfer these eigenspaces back to form two smooth line fields on each $U_P$.  Whenever $R \in U_P \cap U_Q$, the line fields for $U_P$ and $U_Q$ will agree at $R$, since they both describe the extrinsic geometry of $\phi(U_P \cap U_Q)$ at $\phi(R)$.  Hence we can merge all of the ``local'' line fields into two ``global'' line fields $L_1,L_2$ on $\mathbb{D}$.  There is a two-sheeted covering space of $\mathbb{D}$, consisting of all the unit vectors in $L_1$, and since $\mathbb{D}$ is simply-connected, this covering space is not connected.  In other words, there are two unit vector fields $\pm e_1$ on $\mathbb{D}$ which everywhere point along $L_1$.  We define $e_2$ to be the result of rotating $e_1$ a quarter-turn counter-clockwise (i.e., $+\pi/2$ radians).

This allows us to define $\bar{e}_1 := \phi_*(e_1)$ and $\bar{e}_2 := \phi_*(e_2)$ on each $\bar{U}_P$, so we now have a (positively oriented) moving frame $\{ \bar{e}_1, \bar{e}_2, \bar{e}_3 \}$ on $\bar{U}_P$.
As in
\cite[page 82]{doC1994},
we can extend this moving frame to a small open subset $W_P$ of $\mathbb{R}^3$ by displacing frames normally by a sufficiently small amount (choose $\nu_P > 0$ such that $U_P \times (-\nu_P,\nu_P) \rightarrow W_P$ by $(Q,t) \mapsto \bar{Q} + t \cdot [\bar{e}_3]_{\bar{Q}}$ is a diffeomorphism).  See Figure \ref{F:detail} for an illustration.

Since $d\bar{e}_1, d\bar{e}_2, d\bar{e}_3$ are linear transformations on $\mathbb{R}^3$ at every point of $W_P$
(in other words, they are $\mathbb{R}^3$-valued $1$-forms, as defined in
\cite[page 546]{SI1979}, 
while $\bar{e}_1, \bar{e}_2, \bar{e}_3$ are $\mathbb{R}^3$-valued $0$-forms), we can define the 
\emph{connection $1$-forms} $\bar{\omega}^k_j \ (1 \le j,k \le 3)$ by

\begin{equation} \label{Equ:conndef}
d\bar{e}_j(v) = \sum_{k=1}^3 \bar{\omega}^k_j(v)\  \bar{e}_k, 
\end{equation}
\par\noindent
for all vector fields $v$ on $W_P$
\cite[page 78]{doC1994}
\cite[page 287]{SII1979}.
The notation $\bar{\omega}_j^k$ follows
\cite[page 287]{SII1979} and corresponds to $\omega_{jk}$ in
\cite [page 78]{doC1994}.
\par\noindent
On each $W_P$ we can also define $1$-forms $\bar{\theta}^1, \bar{\theta}^2, \bar{\theta}^3$ 
dual to the frame, that is,
\begin{equation} \label{Equ:edual}
\bar{\theta}^i(\bar{e}_j) = \delta^i_j; \quad 1 \le i,j \le 3
\end{equation} 
(using the Kronecker delta).
These $1$-forms on $W_P$ are related to each other by the Cartan structure equations, which follow.

\begin{prop} \label{T:cartanstr}
For all $1 \le i,j \le 3$, and for any moving frame on an open subset of $\mathbb{R}^3$, with connection forms and dual forms defined as in equations \eqref{Equ:conndef} and \eqref{Equ:edual},
\begin{align}
d\bar{\theta}^i &= -\sum_{k=1}^3 \bar{\omega}^i_k \wedge \bar{\theta}^k \label{Equ:1str}\\
d\bar{\omega}^i_j &= -\sum_{k=1}^3 \bar{\omega}^i_k \wedge \bar{\omega}^k_j \label{Equ:2str}
\end{align}
\end{prop}

\begin{proof}
See \cite[page 287]{SII1979} or \cite[page 78]{doC1994}.  The former obtains \eqref{Equ:1str} by expanding $0 = d^2 I$ as a sum of scalar multiples of the $\{ \bar{e}_i \}$, where $I: \mathbb{R}^3 \longrightarrow \mathbb{R}^3$ is the identity map (and $dI = \sum_{i=1}^3 \bar{\theta}^i  \wedge \bar{e}_i$, here an equation of $\mathbb{R}^3$-valued $1$-forms).  Similarly, \eqref{Equ:2str} comes from $0 = d^2 \bar{e}_j$.
\end{proof}
%
\begin{figure}
\begin{overpic}[height=1.75in]{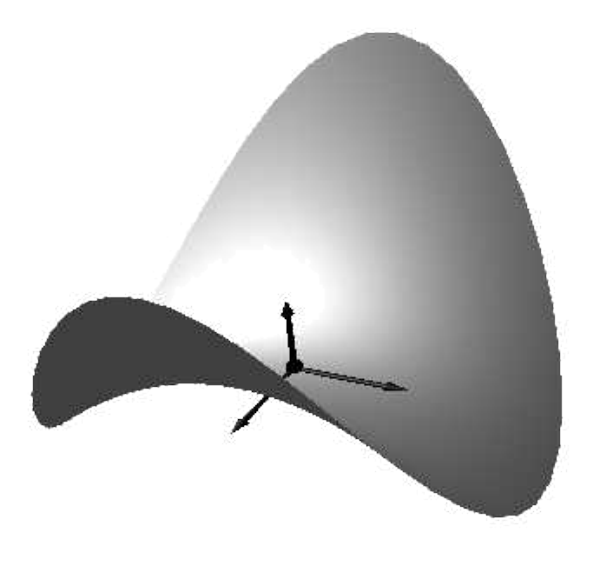}
\put(31,15){$\bar{e}_1$}
\put(68,31){$\bar{e}_2$}
\put(45,47){$\bar{e}_3$}
\put(53,34){$\bar{P}$}
\end{overpic}
\hfil
\includegraphics[height=1.75in]{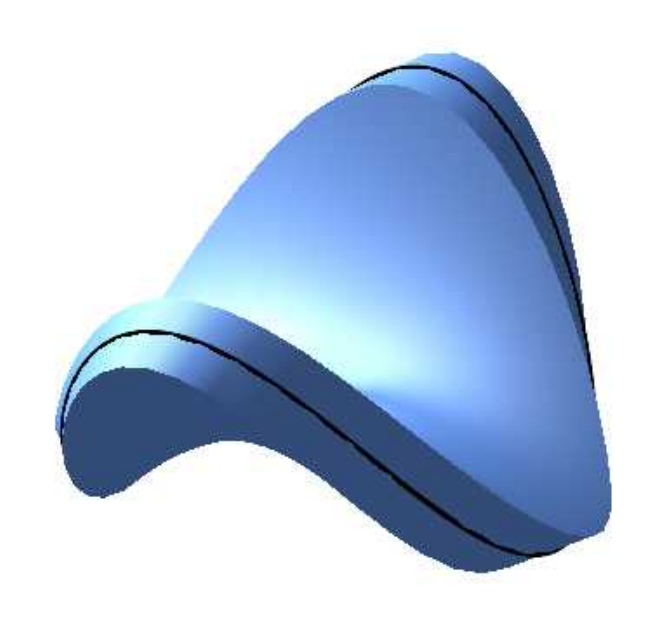}
\caption{At left, $\bar{U}_P = \phi(U_P)$ and frame at $\bar{P} = \phi(P)$; at right, $\bar{U}_P \subset W_P$.} \label{F:detail}
\end{figure}
Now, on each $U_P$, we can define 
the $1$-forms
$\theta^i := \phi^* \bar{\theta}^i$, $1 \le i \le 3$ (noting that $\bar{\theta}^3$ pulls back to the zero $1$-form)
and
$\omega^i_j := \phi^* \bar{\omega}^i_j$, $1 \le i,j \le 3$.  The reader can easily check that $\theta^1, \theta^2$ are dual to $e_1, e_2$.

\begin{rmk}
Orthonormality of the moving frame $\{ \bar{e}_1, \bar{e}_2, \bar{e}_3 \}$ is, in fact, not required for equations \eqref{Equ:1str} and \eqref{Equ:2str}; linear independence of the frame field suffices.  However, orthonormality implies the relations $\bar{\omega}^j_i = -\bar{\omega}^i_j$, hence $\omega^j_i = -\omega^i_j$, 
and in particular, $\omega^i_i = 0$
\cite[page 292]{SII1979} \cite[page 78]{doC1994}.
\end{rmk}

Since $\{ e_1, e_2 \}$ are globally defined on $\mathbb{D}$, the moving frames $\{ \bar{e}_1, \bar{e}_2, \bar{e}_3 \}$ at $\bar{R} \in \bar{U}_P \cap \bar{U}_Q$, arising respectively from $\phi |_{U_P}$ and $\phi |_{U_Q}$, will agree.  Hence the $1$-forms
$\theta^i$ and $\omega^i_j$, defined above, will agree on $U_P \cap U_Q$.  Consequently, as we did above with the line fields, we merge the ``local'' $\theta$'s and $\omega$'s into ``global'' smooth $1$-forms that are defined on $\mathbb{D}$.

Furthermore, since the ``unbarred'' versions of equations \eqref{Equ:1str} and \eqref{Equ:2str} can be justified by basic facts about pullbacks and wedge products in each $U_P$, Proposition \ref{T:cartanstr} implies the following equations on $\mathbb{D}$ (or more precisely, in the exterior algebra of $C^\infty$ differential forms on $\mathbb{D}$):

\begin{equation}\label{Equ:ddual}
\begin{cases}
d\theta^1 &= -(\omega^1_1 \wedge \theta^1 + \omega^1_2 \wedge \theta^2 + \omega^1_3 \wedge \theta^3)
= -0 \wedge \theta^1 - \omega^1_2 \wedge \theta^2 - \omega^1_3 \wedge 0 = \omega_1^2 \wedge \theta^2 \\
d\theta^2 &=  -(\omega^2_1 \wedge \theta^1 + \omega^2_2 \wedge \theta^2 + \omega^2_3 \wedge \theta^3)
= \dots = -\omega_1^2 \wedge \theta^1
\end{cases}
\end{equation}

\begin{equation}\label{Equ:dconn}
\begin{cases}
d\omega_1^2 &= -(\omega_1^2 \wedge \omega_1^1 + \omega_2^2 \wedge \omega_1^2 + \omega_3^2 \wedge \omega_1^3)
= \omega_2^3 \wedge \omega_1^3 = -\omega_1^3 \wedge \omega_2^3 \\
d\omega_1^3 &= -(\omega_1^3 \wedge \omega_1^1 + \omega_2^3 \wedge \omega_1^2 + \omega_3^3 \wedge \omega_1^3)
= -\omega_2^3 \wedge \omega_1^2 = \omega_1^2 \wedge \omega _2^3 \\
d\omega_2^3 &= -(\omega_1^3 \wedge \omega_2^1 + \omega_2^3 \wedge \omega_2^2 + \omega_3^3 \wedge \omega_2^3)
= \omega_1^3 \wedge \omega_1^2 = -\omega_1^2 \wedge \omega_1^3
\end{cases}
\end{equation}

Also, by definition, in each $\bar{U}_P$,

\begin{equation} \label{Equ:eigen}
d\bar{e}_3(\bar{e}_1) = -\bar{\kappa}_1 \bar{e}_1
\textrm{ and }
d\bar{e}_3(\bar{e}_2) = -\bar{\kappa}_2 \bar{e}_2
\end{equation}
\noindent
\cite[page 179]{BL2010},
\cite[pages 125--129]{MP1977},
\cite[pages 103--106]{SII1979}.
From the first equation
we have, using \eqref{Equ:conndef},  
$\bar{\omega}^1_3(\bar{e}_1) = -\bar{\kappa}_1$ and $\bar{\omega}^2_3(\bar{e}_1) = 0$;
similarly, the second equation implies
$\bar{\omega}^1_3(\bar{e}_2) = 0$ and $\bar{\omega}^2_3(\bar{e}_2) = -\bar{\kappa}_2$. 
We conclude, after pullback to $\mathbb{D}$, with $\kappa_i := \bar{\kappa}_i \circ \phi\ (i=1,2)$, that 

\begin{equation} \label{Equ:princurv}
\omega_3^1 = -\kappa_1 \theta^1  \textrm{  and  }  \omega_3^2 = -\kappa_2 \theta^2,
\textrm{ so }
\omega_1^3 = \kappa_1 \theta^1  \textrm{  and  }  \omega_2^3 = \kappa_2 \theta^2.
\end{equation}

\begin{rmk}
To keep the signs straight, it is useful to keep in mind the example of the unit sphere (with one point removed, so that there exists a moving frame), oriented with outward normal.  Then $\bar{e}_3$ equals $p$ at $p \in S^2$, so $d\bar{e}_3$ is the identity map.  The normal curvatures at each point all equal $-1$ (negative since curves on the sphere bend \emph{away} from the normal vector).  So for this surface the principal curvatures both equal $-1$, which is consistent with equation \eqref{Equ:eigen}.
\end{rmk}

\subsection{Expressions for $\omega_1^2$ in terms of  $\{ \theta^1, \theta^2 \}$} \label{SS:omegatheta}

We now turn to the proofs of two key lemmas, which relate $\omega_1^2$ to the principal curvatures and to the angle between the asymptotic directions and $e_1$.

\begin{lem} \label{L:connprinequ}
With $\{ e_1, e_2 \}$, $\{ \theta^1, \theta^2 \}$, and $\omega_1^2, \kappa_1, \kappa_2$ defined as above,
\begin{equation} \label{Equ:connprin}
(\kappa_1 - \kappa_2) \omega_1^2 = d\kappa_1(e_2) \theta^1 + d\kappa_2(e_1) \theta^2.
\end{equation}
\end{lem}

\begin{proof} 
To verify this equation of $1$-forms, it suffices to show the following two equations:
\begin{equation} \label{Equ:ate1ate2}
(\kappa_1 - \kappa_2) \omega_1^2 (e_1) \overset{?}{=} d\kappa_1(e_2) \textrm{  and  }
(\kappa_1 - \kappa_2) \omega_1^2 (e_2) \overset{?}{=} d\kappa_2(e_1)
\end{equation}
We rewrite $d\omega_1^3$ in two ways, using equations \eqref{Equ:ddual}, \eqref{Equ:dconn}, and \eqref{Equ:princurv}, obtaining
\[
\kappa_1 \omega_1^2 \wedge \theta^2 + d\kappa_1 \wedge \theta^1 
\overset{\eqref{Equ:ddual}}{=}
\kappa_1 d\theta^1 + d\kappa_1 \wedge \theta^1 
\overset{\eqref{Equ:princurv}}{=} 
d\omega_1^3 
\overset{\eqref{Equ:dconn}}{=} 
\omega_1^2 \wedge \omega_2^3 
\overset{\eqref{Equ:princurv}}{=} 
\kappa_2 \omega_1^2 \wedge \theta^2.
\]
\par\noindent
In other words, 
$(\kappa_1 - \kappa_2) \omega_1^2 \wedge \theta^2 = -d\kappa_1 \wedge \theta^1$.
Evaluating both sides of this equation on the pair ($e_1$,$e_2$), we obtain
\begin{align*}
(\kappa_1 - \kappa_2) \omega_1^2 \wedge \theta^2(e_1,e_2) &= 
(\kappa_1 - \kappa_2) (\omega_1^2(e_1) \ \theta^2(e_2) - \omega_1^2(e_2) \ \theta^2(e_1))
= (\kappa_1 - \kappa_2) \omega_1^2(e_1) \\
\textrm{  and  } -d\kappa_1 \wedge \theta^1(e_1,e_2) &=
-d\kappa_1(e_1) \ \theta^1(e_2) + d\kappa_1(e_2) \ \theta^1(e_1)
= d\kappa_1(e_2),
\end{align*}
\par\noindent
which justifies the first half of \eqref{Equ:ate1ate2}.  
Similarly, rewriting $d\omega_2^3$,
\begin{equation*}
-\kappa_2 \omega_1^2 \wedge \theta^1 + d\kappa_2 \wedge \theta^2 
\overset{\eqref{Equ:ddual}}{=}
\kappa_2 d\theta^2 + d\kappa_2 \wedge \theta^2 
\overset{\eqref{Equ:princurv}}{=} 
d\omega_2^3 
\overset{\eqref{Equ:dconn}}{=} 
-\omega_1^2 \wedge \omega_1^3 
\overset{\eqref{Equ:princurv}}{=} 
-\kappa_1 \omega_1^2 \wedge \theta^1.
\end{equation*}
\par\noindent
In other words, $-(\kappa_1 - \kappa_2) \omega_1^2 \wedge \theta^1 = d\kappa_2 \wedge \theta^2$.  Again evaluating the left-hand and right-hand sides at ($e_1$,$e_2$),
\begin{align*}
-(\kappa_1 - \kappa_2) \omega_1^2 \wedge \theta^1(e_1,e_2) &= 
\dots
= (\kappa_1 - \kappa_2) \omega_1^2(e_2) \\
\textrm{  and  } d\kappa_2 \wedge \theta^2(e_1,e_2) &=
\dots
= d\kappa_2(e_1),
\end{align*}
\par\noindent
which justifies the second half of \eqref{Equ:ate1ate2}.
\end{proof}

For all $P \in \mathbb{D}$, the normal curvatures at $\phi(P)$ can be described by 
$\sff(v) := \langle -d\bar{e}_3(\phi_*(v)), \phi_*(v) \rangle$,
defined in terms of the Euclidean inner product for vectors in $\mathbb{R}^3$
\cite[page 87]{doC1994}.
In essence, this is the second fundamental form
for the immersed surface near $\phi(P)$, pulled back to a quadratic form on $T_P(\mathbb{D})$.
Using Euler's theorem
\cite[page 147]{doC2016}
\cite[page 129]{MP1977},
and (for the first time) making use of the fact that Gaussian curvature $K := \kappa_1 \kappa_2$ equals $-1$ everywhere on $\mathbb{D}$, we have normal curvature zero when
\[
0 = \sff(v) = \kappa_1 \cos^2\alpha + \kappa_2 \sin^2\alpha = \kappa_1 \cos^2\alpha - (1/\kappa_1) \sin^2\alpha
\]
\par\noindent
for a unit vector $v$ that makes an angle $\pm\alpha$ with $e_1$.
We conclude that $\tan^2 \alpha = \kappa_1^2$, so set $\alpha := \arctan(\kappa_1) \in (0,\pi/2)$, and define the asymptotic directions
\begin{equation} \label{Equ:defasympt}
E_1 := \cos \alpha \cdot e_1 - \sin \alpha \cdot e_2 \textrm{  and  } E_2 := \cos \alpha \cdot e_1 + \sin \alpha \cdot e_2.
\end{equation}
\par\noindent
It follows that
\begin{equation} \label{Equ:esfromEs}
e_1 = \frac{1}{2} \sec\alpha (E_1+ E_2) 
\textrm{  and  }
e_2 = \frac{1}{2} \csc\alpha (-E_1+E_2).
\end{equation}

\begin{lem} \label{L:connalpha}
With $\{ e_1, e_2 \}$, $\{ \theta^1, \theta^2 \}$, and $\omega_1^2, \alpha$ defined as above,
\begin{equation}
\omega_1^2 = \tan\alpha\, d\alpha(e_2)\cdot \theta^1 + \cot\alpha\, d\alpha(e_1) \cdot \theta^2 \label{Equ:conntancot}
\end{equation}
\end{lem}

\begin{proof}
Since $\kappa_2 = -1/\kappa_1 = -1/\tan \alpha = -\cot \alpha$,
\[
\kappa_1 - \kappa_2 = \frac{\sin \alpha}{\cos \alpha} + \frac{\cos \alpha}{\sin \alpha} = \frac{1}{\cos \alpha\, \sin \alpha}.
\]
\par\noindent
Therefore, using $d\kappa_1 = d(\tan \alpha) = \sec^2 \alpha\, d\alpha$ and 
$d\kappa_2 = d(-\cot \alpha) = \csc^2 \alpha\, d\alpha$,
\begin{align}
\frac{d\kappa_1}{\kappa_1 - \kappa_2} &=  (\sec^2 \alpha\, d\alpha)(\cos \alpha\, \sin \alpha) = \tan \alpha \, d\alpha  \label{Equ:dk1}\\
\frac{d\kappa_2}{\kappa_1 - \kappa_2} &=  (\csc^2 \alpha\, d\alpha)(\cos \alpha\, \sin \alpha) = \cot \alpha\, d\alpha \label{Equ:dk2}
\end{align}
\par\noindent
After evaluating \eqref{Equ:dk1} at $e_2$ and \eqref{Equ:dk2} at $e_1$, the desired conclusion follows from Lemma \ref{L:connprinequ}.
\end{proof}

\subsection{Use $\{ E_1, E_2 \}$ to construct a global coordinate map} \label{SS:F}

We now turn our attention to the properties of the asymptotic vector fields $E_1,E_2$.

\begin{lem} \label{L:closed}
Let $\{ \eta^1, \eta^2 \}$ be the $1$-forms which are dual to $\{ E_1, E_2 \}$.  Then $d\eta^1 = 0$ and $d\eta^2 = 0$.
\end{lem}

\begin{proof}
It is not hard to use the definitions of $E_1$ and $E_2$ in \eqref{Equ:defasympt} to verify that
\begin{equation}  \label{Equ:etafromtheta}
\eta^1 = \frac{1}{2} [\sec\alpha \cdot \theta^1 - \csc\alpha \cdot \theta^2]
\textrm{  and  }
\eta^2 = \frac{1}{2} [\sec\alpha \cdot \theta^1 + \csc\alpha \cdot \theta^2]
\end{equation}

Next, we see that
\begin{align} \label{Equ:etasec}
d(\sec\alpha \cdot \theta^1) &\overset{\eqref{Equ:ddual}}{=} \sec\alpha\, \tan\alpha\, d\alpha \wedge \theta^1 + \sec\alpha (\omega_1^2 \wedge \theta^2) \notag\\
&=  \sec\alpha\, \tan\alpha\, (d\alpha(e_1)\cdot \theta^1 + d\alpha(e_2)\cdot \theta^2) \wedge \theta^1 + \sec\alpha (\omega_1^2 \wedge \theta^2) \notag\\
&= -\sec\alpha\, \tan\alpha\, d\alpha(e_2)\cdot \theta^1 \wedge \theta^2 + \sec\alpha (\omega_1^2 \wedge \theta^2) \notag\\
&\overset{\eqref{Equ:conntancot}}{=}  -\sec\alpha\, \tan\alpha\, d\alpha(e_2)\cdot \theta^1 \wedge \theta^2 \notag\\ 
&\qquad +\sec\alpha ((\tan\alpha\, d\alpha(e_2)\cdot \theta^1 + \cot\alpha\, d\alpha(e_1) \cdot \theta^2) \wedge \theta^2) \notag\\
&=  -\sec\alpha\, \tan\alpha\, d\alpha(e_2)\cdot \theta^1 \wedge \theta^2 + \sec\alpha (\tan\alpha\, d\alpha(e_2)\cdot \theta^1 \wedge \theta^2) = 0
\end{align}
\par\noindent
and similarly
\begin{align} \label{Equ:etacosec}
d(\csc\alpha \cdot \theta^2) &\overset{\eqref{Equ:ddual}}{=} -\csc\alpha\, \cot\alpha\, d\alpha \wedge \theta^2 + \csc\alpha (-\omega_1^2 \wedge \theta^1) \notag\\
&=  -\csc\alpha\, \cot\alpha\, (d\alpha(e_1)\cdot \theta^1 + d\alpha(e_2)\cdot \theta^2) \wedge \theta^2 + \csc\alpha (-\omega_1^2 \wedge \theta^1) \notag\\
&= -\csc\alpha\, \cot\alpha\, d\alpha(e_1)\cdot \theta^1 \wedge \theta^2 - \csc\alpha (\omega_1^2 \wedge \theta^1) \notag\\
&\overset{\eqref{Equ:conntancot}}{=}  -\csc\alpha\, \cot\alpha\, d\alpha(e_1)\cdot \theta^1 \wedge \theta^2 \notag\\
&\qquad - \csc\alpha ((\tan\alpha\, d\alpha(e_2)\cdot \theta^1 + \cot\alpha\, d\alpha(e_1) \cdot \theta^2) \wedge \theta^1) \notag\\
&=  -\csc\alpha\, \cot\alpha\, d\alpha(e_1)\cdot \theta^1 \wedge \theta^2 - \csc\alpha (\cot\alpha\, d\alpha(e_1)\cdot (-\theta^1 \wedge \theta^2)) = 0 
\end{align}

By subtracting and adding equations \eqref{Equ:etasec} and \eqref{Equ:etacosec}, and then comparing to equation \eqref{Equ:etafromtheta}, we conclude that $d\eta^1 = 0$ and $d\eta^2 = 0$.
\end{proof}

We next show that the vector fields $E_1$ and $E_2$ on $\mathbb{D}$ each generate a global flow (i.e., each integral curve is defined for all $t \in \mathbb{R}$).  
We will use the following special case of \cite[Chapter 6, \S 6, Theorem 2]{C1961}.

\begin{thm} \label{T:existlocalsoln}
Let $a,b$ be positive numbers, let $t_0 \in \mathbb{R}$, let $\mathbf{y_0} \in \mathbb{R}^2$, and let 
$R := \{ (t,\mathbf{y}) : |t - t_0| \le a, |\mathbf{y}-\mathbf{y_0}| \le b \}$.  
If
\begin{enumerate}
\item $f: R \rightarrow \mathbb{R}^2$ is a continuous function, and 
\item there exists $L > 0$ such that
for all $(t,\mathbf{y_1}), (t,\mathbf{y_2}) \in R$, $| f(t,\mathbf{y_1}) - f(t,\mathbf{y_2}) | \le L |\mathbf{y_1} - \mathbf{y_2}|$, and
\item there exists $M > 0$ such that for all $(t,\mathbf{y}) \in R$, $|f(t,\mathbf{y})| \le M$,
\end{enumerate}
then
the initial-value problem $d\mathbf{y}/dt = f(t,\mathbf{y}),\ \mathbf{y}(t_0)=\mathbf{y_0}$ has a solution 
\[
\mathbf{y} = \gamma(t), \ 
\gamma: (t_0-\epsilon, t_0+\epsilon) \rightarrow  \{ \mathbf{y} \in \mathbb{R}^2 : |\mathbf{y}-\mathbf{y_0}| \le b \},
\]
where $\epsilon := \min ( a, b/M )$. \qed
\end{thm}

Note that $\epsilon$ does not depend on $L$, the Lipschitz constant.  

\begin{prop} \label{P:globalflow}
If $E$ is a $C^\infty$ vector field on $\mathbb{D}$ which everywhere has unit length with respect to the hyperbolic metric, then
every maximal integral curve of $E$ is defined for all $t \in \mathbb{R}$.
\end{prop}

\begin{proof}
Preparing to apply Theorem \ref{T:existlocalsoln}, we will think of the unit-hyperbolic-length vector field $E$ as a vector field on the open unit disk in $\mathbb{R}^2$, and now measure length with respect to the Euclidean metric.
Set $t_0 := 0$ and $\mathbf{y_0} := (0,0) = O$.  Next, set $a := 1$, and set $b := 1/2$, so that
$R = \{ |t| \le 1, |\mathbf{y}| \le 1/2 \}$, and define $f(t,\mathbf{y}) := E|_\mathbf{y}$.
By the definition of the hyperbolic metric on $\mathbb{D}$ (see Appendix), we can take $M = 1/2$ (for any choice of $b \in (0,1)$).  Hence, $\epsilon = \min (1, \frac{1}{2}/\frac{1}{2}) = 1$.

The partial derivatives of $f$ with respect to the standard coordinates of $\mathbb{R}^2$ are continuous on the compact set $R$, and hence are bounded, so we conclude that the Lipschitz constant $L$ exists.  By Theorem \ref{T:existlocalsoln}, we have a solution
$\gamma: (-1,1) \rightarrow \mathbb{D} \subset \mathbb{R}^2$.

Now fix an arbitrary point $P \in \mathbb{D}$.  There exists a hyperbolic isometry $\psi: \mathbb{D} \rightarrow \mathbb{D}$ which maps $P$ to $O$ (see equation \eqref{Equ:PtofromO}).
Use $\psi_*$ to map the vector field $E$ to a ``recentered'' vector field $\tilde{E}$ on $\mathbb{D}$, then apply the previous reasoning to $\tilde{E}$, obtaining an integral curve $\gamma: (-1,1) \rightarrow \mathbb{D}$ such that $\gamma(0) = O$.  The curve $\psi^{-1} \circ \gamma: (-1,1) \rightarrow \mathbb{D}$ will be the desired integral curve through $P$.

Since $E$ is a smooth vector field on the smooth manifold $\mathbb{D}$, and every point of $\mathbb{D}$ can flow forward and backward along $E$ for at least one unit of time, it follows from the Uniform Time Lemma \cite[Lemma 9.15, page 216]{Lee2013} that every integral curve can be extended to an integral curve defined on $(-\infty,\infty)$.
\end{proof}

\begin{cor} \label{C:E1E2flow}
All integral curves of $E_1$ and $E_2$, the unit asymptotic directions, are defined for all time. \qed
\end{cor} 

\begin{prop} \label{P:coords}
There is a $C^\infty$ diffeomorphism $F$ from $\mathbb{D}$ onto 
$\mathbb{R}^2 := \{ (x^1,x^2): x^1,x^2 \in \mathbb{R} \}$, 
such that
\begin{enumerate}
\item $F$ is a (global) coordinate chart for $\mathbb{D}$, and
\item $F_*(E_i) = \frac{\partial}{\partial x^i}$, $i=1,2$.
\end{enumerate}
\end{prop}

\begin{proof}
Working in $\mathbb{D}$, with $O = (0,0)$, define $F(P) := (\int_O^P \eta^1, \int_O^P \eta^2)$.  $F$ is well-defined, using Green's theorem, since $\eta^1$ and $\eta^2$ are closed and any two paths from $O$ to $P$ bound a region in $\mathbb{D}$.  By definition, the differential of $F$ at $P$ will map $(E_i)_P$ to $(\partial/\partial x^i)_{F(P)}$, for $i=1,2$.  Since $E_1$ and $E_2$ are everywhere linearly independent, the differential is everywhere nonsingular, so $F$ is a local $C^\infty$ diffeomorphism (using a $C^\infty$ version of the Inverse Function theorem; see
\cite[Theorem C.34]{Lee2013}
).

In addition, $F$ is surjective, since in order to find a point in $\mathbb{D}$ which maps to $(a,b) \in \mathbb{R}^2$, it suffices to start at $O$ at time zero, follow an integral curve for $E_1$ (forward if $a>0$, backward if $a<0$) for time $|a|$, then similarly follow an integral curve for $E_2$ for time $|b|$, using Corollary \ref{C:E1E2flow}.  By the definition of $F$, along the first part of the curve, the second coordinate of the output will not change, and vice versa.

With respect to the orthonormal bases $\{ e_1,e_2 \}$ in $\mathbb{D}$ and $\{ \partial/\partial x^1, \partial/\partial x^2 \}$ in $\mathbb{R}^2$, any local inverse for $F$ has Jacobian matrix 
\begin{equation} \label{Equ:GJacob}
M_\alpha := \begin{bmatrix} \cos \alpha & \cos \alpha \\ -\sin \alpha & \sin \alpha \end{bmatrix},
\end{equation}
(using equation \eqref{Equ:defasympt}) and the singular values of this matrix are easily calculated to be $\sqrt{2}\cos \alpha$ and $\sqrt{2}\sin\alpha$.  From this we can conclude that all local inverses are Lipschitz maps with constant $\sqrt{2}$.  Now let $\gamma: [0,1] \rightarrow \mathbb{R}^2$ be a continuous path, and let $Q$ be a point in $\mathbb{D}$ such that $F(Q) = \gamma(0)$.  We claim that $\gamma$ has a unique lift to $Q$; that is, there is a unique path $\tilde{\gamma}: [0,1] \rightarrow \mathbb{D}$ such that $\tilde{\gamma}(0)=Q$ and such that $F \circ \tilde{\gamma} = \gamma$.

Let $\mathcal{T} := \{ t \in [0,1] : \gamma|_{[0,t]} \textrm{ has a unique lift} \}$.  This is an open subset of $[0,1]$, since when $\gamma$ lifts as far as $[0,s]$, a local diffeomorphism from an open subset of $\mathbb{D}$ to an open neighborhood of $\gamma(s)$ will allow the lift to be (uniquely) extended a little further.  It is also true that $\mathcal{T}$ is a closed subset of $[0,1]$, since if $\gamma|_{[0,s)}$ lifts, then $\gamma$ is uniformly continuous on $[0,s) \subset [0,1]$.  Using the Lipschitz bound in the previous paragraph, $\tilde{\gamma}|_{[0,s)}$ is also uniformly continuous, and hence (see Appendix for details) $\tilde{\gamma}$ has a unique extension to $[0,s]$.  Finally, since $[0,1]$ is connected, and $\mathcal{T}$ is a non-empty subset that is both open and closed, we conclude that $\mathcal{T} = [0,1]$, hence $\gamma$ lifts uniquely.

Since $\mathbb{D}$ is path-connected, and $\mathbb{R}^2$ is simply-connected, the following lemma applies to the map $F: \mathbb{D} \rightarrow \mathbb{R}^2$.

\begin{lem} [{\cite[Proposition 6.12]{L2003}}]
Let $f: X \rightarrow Y$ be a local homeomorphism with the unique-path-lifting property.  If $X$ is path-connected and $Y$ is simply-connected, then $f$ is a homeomorphism.
\end{lem}

\begin{proof}
(Sketch) Construct a path between any two points in $X$ that $f$ maps to the same place.  Use $f$ to map the path, continuously deform the mapped path to a constant path, and then lift the homotopy, to show that the initial and final points of the original path are equal.
\end{proof}

Since $F$ is now known to be a homeomorphism (and in particular is injective), and has already been shown to be a local $C^\infty$ diffeomorphism, it follows that $F$ is a $C^\infty$ diffeomorphism.
\end{proof}

\begin{cor} \label{C:Finverse}
Let $G: \mathbb{R}^2 \rightarrow \mathbb{D}$ be defined by mapping the point $(x^1,x^2) \in \mathbb{R}^2$ to the point in $\mathbb{D}$ obtained by starting at $O$, following an integral path for the vector field $E_1$ for time $x^1$, then following an integral path for the vector field $E_2$ for time $x^2$ (moving backward when these numbers are negative).  Then $G = F^{-1}$.
\end{cor}

\begin{proof}
By construction, $F \circ G$ is the identity map on $\mathbb{R}^2$.  Proposition \ref{P:coords} implies that $F$ is injective, so $F \circ I_\mathbb{D} = I_{\mathbb{R}^2} \circ F = (F \circ G) \circ F = F \circ (G \circ F)$ implies $I_\mathbb{D} = G \circ F$.
\end{proof}

Since Proposition \ref{P:coords} has shown that $F$ is a coordinate mapping for $\mathbb{D}$, we can define coordinate functions $u^1, u^2$ on $\mathbb{D}$ such that $u^1 = x^1 \circ F,\ u^2 = x^2 \circ F$ and
\begin{equation} \label{Equ:usandEs}
\frac{\partial}{\partial u^1} = E_1,\ \frac{\partial}{\partial u^2} = E_2;
\textrm{ hence }
du^1 = \eta^1,\ du^2 = \eta^2.
\end{equation}
\par\noindent
(By a common abuse of notation, we let $x^1$ and $x^2$ denote the functions on $\mathbb{R}^2$ that return, respectively, the first and second coordinates of a point.)

\subsection{Area computations} \label{SS:area}

Adding and subtracting the equations in \eqref{Equ:etafromtheta}, we have
\begin{equation} \label{Equ:thetafrometa}
\theta^1 = \cos\alpha (\eta^1 + \eta^2) \textrm{  and  } \theta^2 = \sin\alpha (-\eta^1 + \eta^2).
\end{equation}
\par\noindent
Therefore, using Lemma \ref{L:connalpha},
\begin{align}
\omega_1^2 &= \tan\alpha\, d\alpha(e_2)\cdot \theta^1 + \cot\alpha\, d\alpha(e_1) \cdot \theta^2 \notag\\
&\overset{\eqref{Equ:thetafrometa}}{=} \tan\alpha\, d\alpha(e_2)\cdot \cos\alpha (\eta^1 + \eta^2) + \cot\alpha\, d\alpha(e_1) \cdot \sin\alpha (-\eta^1 + \eta^2) \notag\\
&= \sin\alpha\, d\alpha(e_2) \cdot (\eta^1 + \eta^2) + \cos\alpha\, d\alpha(e_1) \cdot (-\eta^1 + \eta^2) \notag\\
&= (\sin\alpha\, d\alpha(e_2) - \cos\alpha\, d\alpha(e_1)) \eta^1
+ (\sin\alpha\, d\alpha(e_2) + \cos\alpha\, d\alpha(e_1)) \eta^2 \label{Equ:omegarewrite}
\end{align}
\par\noindent
while also (rewrite \eqref{Equ:esfromEs} using \eqref{Equ:usandEs}, apply $d\alpha$ to both sides of each equation, and simplify)
\[
\sin\alpha\, d\alpha(e_2) = \frac{1}{2} \left( -\frac{\partial \alpha}{\partial u^1} + \frac{\partial \alpha}{\partial u^2} \right)
\textrm{  and  }
\cos\alpha\, d\alpha(e_1) = \frac{1}{2} \left( \frac{\partial \alpha}{\partial u^1} + \frac{\partial \alpha}{\partial u^2} \right).
\]

We define $\Theta := 2\alpha$, the angle between $E_1$ and $E_2$ (using the capital letter to avoid confusion with the 1-forms $\theta^1, \theta^2$).  Now \eqref{Equ:omegarewrite} becomes
\begin{equation} \label{Equ:conneta}
\omega_1^2 = -\frac{\partial \alpha}{\partial u^1} \eta^1 + \frac{\partial \alpha}{\partial u^2} \eta^2
= -\frac{1}{2} \frac{\partial \Theta}{\partial u^1} \eta^1 + \frac{1}{2} \frac{\partial \Theta}{\partial u^2} \eta^2,
\end{equation}
\par\noindent
which means that (using \eqref{Equ:etafromtheta} and $\frac{1}{2} \sec \alpha \csc \alpha = \csc \Theta$ to show the final equality)
\begin{align} \label{Equ:dconnfirst}
d\omega_1^2 &= -\frac{1}{2} \frac{\partial^2 \Theta}{\partial u^2 \partial u^1} \eta^2 \wedge \eta^1
+ \frac{1}{2} \frac{\partial^2 \Theta}{\partial u^1 \partial u^2} \eta^1 \wedge \eta^2
=  \frac{\partial^2 \Theta}{\partial u^1 \partial u^2} (\eta^1 \wedge \eta^2) \notag\\
&= \csc\Theta \frac{\partial^2 \Theta}{\partial u^1 \partial u^2} (\theta^1 \wedge \theta^2),
\end{align}
\par\noindent
where $\theta^1 \wedge \theta^2$ represents the ``area form'' $dA$ on $\mathbb{D}$ (since $\{e_1, e_2 \}$ is orthonormal).
At the same time,
\begin{equation} \label{Equ:dconnsecond}
d\omega_1^2  \overset{\eqref{Equ:dconn}}{=} -\omega_1^3 \wedge \omega_2^3
\overset{\eqref{Equ:princurv}}{=} -(\kappa_1 \theta^1) \wedge (\kappa_2 \theta^2) 
= -\kappa_1 \kappa_2 (\theta^1 \wedge \theta^2)
= \theta^1 \wedge \theta^2.
\end{equation}
\par\noindent
Comparing equations \eqref{Equ:dconnfirst} and \eqref{Equ:dconnsecond},
we have the PDE 
\begin{equation} \label{Equ:KeyPDEinD}
\frac{\partial^2 \Theta}{\partial u^1 \partial u^2} = \sin\Theta,
\end{equation}
\par\noindent
on $\mathbb{D}$, which we can transfer to the $x^1x^2$-plane, after defining $\hat{\Theta} := \Theta \circ F^{-1}$, as
\begin{equation} \label{Equ:KeyPDEinR2}
\frac{\partial^2 \hat{\Theta}}{\partial x^1 \partial x^2} = \sin \hat{\Theta}.
\end{equation}

By Proposition \ref{P:coords} and the definition of $\Theta$, $\det(M_\alpha) \overset{\eqref{Equ:GJacob}}{=}  \sin(2\alpha) = \sin\Theta$ is the factor by which area is distorted, when $F^{-1}$ maps the $x^1x^2$-plane to $\mathbb{D}$.
Hence, the area $A$ of the hyperbolic plane is given by
 \begin{align} \label{Equ:finitearea}
 A &= \lim_{a \to \infty} \int_{-a}^a \int_{-a}^a \sin \hat{\Theta}\, dx^1\, dx^2 
= \lim_{a \to \infty} \int_{-a}^a \int_{-a}^a \frac{\partial^2 \hat{\Theta}}{\partial x^1 \partial x^2}\, dx^1\, dx^2 \notag\\
& = \int_{-a}^a \left[\frac{\partial \hat{\Theta}}{\partial x^2}\right]_{x^1=-a}^{x^1=a} \, dx^2 
= \int_{-a}^a [\frac{\partial \hat{\Theta}}{\partial x^2} (a,x^2) - \frac{\partial \hat{\Theta}}{\partial x^2} (-a,x^2)] \, dx^2 \notag\\
& =  \lim_{a \to \infty} [\hat{\Theta}(a,a) - \hat{\Theta}(-a,a) - \hat{\Theta}(a,-a) + \hat{\Theta}(-a,-a)],
 \end{align}
\par\noindent
and since $0 < \Theta < \pi$, $A$ is bounded above by $2\pi$.  On the other hand, the hyperbolic plane has infinite area, as computed below (in the Poincar\'e model; see the Appendix for the Riemannian metric):
\begin{align} \label{Equ:infinitearea}
\lim_{t \to 1^-} \int_0^{2\pi} \int_0^t \frac{4}{(1-r^2)^2} \, r \, dr \, d\theta &=
\lim_{t \to 1^-} \int_0^{2\pi} \left[\frac{2}{(1-r^2)} \right]_0^t  \, d\theta  \notag\\
&=  2\pi \cdot  \lim_{t \to 1^-} \left( \frac{2}{1-t^2} - 2 \right) = \infty 
\end{align}
\par\noindent
Assuming the existence of $\phi$ has led to a contradiction, hence there does not exist an isometric immersion of $\mathbb{D}$ into $\mathbb{R}^3$, concluding the proof of Theorem \ref{T:Hilbert}.
\end{proof} 

\begin{cor} \label{C:completesurface}
If $S$ is a $C^\infty$ Riemannian $2$-manifold for which the metric is complete and has constant curvature $-1$, then there does not exist a $C^\infty$ isometric immersion of $S$ into $\mathbb{R}^3$.
\end{cor}

\begin{proof}
By \cite[Corollary 2.3.17]{W1977}, the universal cover, $\tilde{S}$, of any such $S$ would be isometric to $\mathbb{H}^2$, and the composition $\tilde{S} \rightarrow S \rightarrow \mathbb{R}^3$ would be a counterexample to Theorem \ref{T:Hilbert}.
\end{proof}

\begin{cor} \label{C:negcurv}
If $S$ is a $C^\infty$ Riemannian $2$-manifold for which the metric is complete and has constant curvature $-k^2$ ($k > 0$), then there does not exist a $C^\infty$ isometric immersion of $S$ into $\mathbb{R}^3$.
\end{cor}

\begin{proof}
Suppose that $\psi: S \rightarrow \mathbb{R}^3$ is such an isometric immersion, and let $h: \mathbb{R}^3 \rightarrow \mathbb{R}^3$ be the dilation $(x,y,z) \mapsto (kx,ky,kz)$.  Since the effect of the dilation is to divide all normal curvatures by $k$, the image of the composition $h \circ \psi$ has constant curvature $-1$ everywhere, which contradicts Corollary \ref{C:completesurface} (if the original inner product at every point on $S$ is multiplied by $k^2$, the new Riemannian metric will still be complete, and $h \circ \psi$ will be an isometric immersion).
\end{proof}

\section{Brief History} \label{Sec:history}

Hilbert's original proof 
\cite{Hi1900}
assumed that the surface in $\mathbb{R}^3$ was real-analytic (and had constant negative curvature $-1$).  However, the outline of his argument will work if $\phi$ is smooth enough to ensure that the PDE in equation
\eqref{Equ:KeyPDEinD} makes sense.
Tracing back the definition of $\Theta$, we can see that if
$\phi \in C^4(\mathbb{D})$,
then the normal vector $\bar{e}_3$ is $C^3$,
and the eigenvalues $\bar{\kappa_1}, \bar{\kappa}_2$ of $-d\bar{e}_3$ are $C^2$,
as is $\Theta = 2 \arctan(\kappa_1)$.
So $\partial^2 \Theta / \partial u^1 \partial u^2$ exists and is continuous
on $\mathbb{D}$.
With more sophisticated techniques, Efimov proved that there is no $C^2$ isometric immersion of the hyperbolic plane; see 
\cite{TKM1972} 
for an exposition and more history.
On the other hand, by work of Nash and Kuiper,
there \emph{does} exist a $C^1$ isometric embedding of the hyperbolic plane in $\mathbb{R}^3$
\cite{N1954} \cite{K1955}.

\section{Appendix: Metric properties of $\mathbb{H}^2$} \label{Sec:appx}

Let $O := (0,0) \in \mathbb{D} := \{ (x,y) \in \mathbb{R}^2 : x^2+y^2<1 \}$.  
$\mathbb{D}$ is the Poincaré disk model for $\mathbb{H}^2$, with Riemannian
metric $g_{11} = g_{22} = 4/(1- x^2 - y^2)^2,\  g_{12} = g_{21} = 0$
\cite[page 179]{MP1977},
\cite[Theorem 4.5.5]{R1994}.
$\mathbb{D}$ inherits an orientation from the usual orientation on $\mathbb{R}^2$.
Furthermore, the corresponding distance function on $\mathbb{D}$ is
\cite[Theorem 4.5.1]{R1994}

\[
d_{\mathbb{D}} (P,Q) := \arccosh \left( 1 + \frac{2|P-Q|^2}{(1-|P|^2)(1-|Q|^2 )} \right)
\]
\par\noindent
where $P,Q \in \mathbb{D}$ and $|P|$ denotes the Euclidean length of the vector $P$.

Proofs of the following properties can be found in many textbooks, including \cite{R1994}.

\begin{enumerate}
\item The metrics $g$ and $d_{\mathbb{D}}$ are invariant under M\"obius transformations which leave $\mathbb{D}$ invariant.
\item The $d_{\mathbb{D}}$-metric topology on $\mathbb{D}$ equals the topology that the $\mathbb{D}$ inherits as a subset of $\mathbb{R}^2$ (with the usual topology).
\item Every open $d_\mathbb{D}$-ball centered at a point in $\mathbb{D}$ equals (as a set) an open Euclidean disk whose closure is contained in $\mathbb{D}$.
\item The metric space $(\mathbb{D},d_{\mathbb{D}})$ is complete
\cite[Theorem 8.5.1]{R1994}.
\item After identifying $(x,y)$ with the complex number $z = x+iy$, the maps
\begin{equation} \label{Equ:PtofromO}
z \mapsto \frac{z - P}{-\bar{P}z + 1}, \ 
z \mapsto \frac{z + P}{\bar{P}z + 1}
\end{equation}
\par\noindent
are isometries of the Riemannian manifold 
$(\mathbb{D},g)$ and inverses of each other,
 taking $P \in \mathbb{D}$ to $O$ and $O$ to $P$, respectively.
\cite[Section 4.5, Exercise 10]{R1994}
\cite[3 IX \& 6 III 11]{Ne1997}.
\end{enumerate}

Furthermore, if $r>0$ and $\gamma: [0,r) \rightarrow \mathbb{D}$ is uniformly continuous, then there is a unique extension of $\gamma$ to a (uniformly) continuous map $\bar{\gamma}: [0,r] \rightarrow \mathbb{D}$.
This is a very special case of an exercise in \cite[Chapter 7]{M2000}.  The proof hinges on the fact that every sequence in $[0,r)$ converging to $r$ is a Cauchy sequence and will be mapped by $\gamma$ to a Cauchy sequence in $\mathbb{D}$.
\newpage

%
%
\bibliographystyle{amsalpha}
\bibliography{Hilbert-A1}

\providecommand{\bysame}{\leavevmode\hbox to3em{\hrulefill}\thinspace}
\providecommand{\MR}{\relax\ifhmode\unskip\space\fi MR }
\providecommand{\MRhref}[2]{%
  \href{http://www.ams.org/mathscinet-getitem?mr=#1}{#2}
}
\providecommand{\href}[2]{#2}
\begin{thebibliography}{Mun00}

\bibitem[BL10]{BL2010}
Thomas Banchoff and Stephen Lovett, \emph{Differential geometry of curves and
  surfaces}, A K Peters, Ltd., Natick, MA, 2010. \MR{2674651}

\bibitem[Cod61]{C1961}
Earl~A. Coddington, \emph{An introduction to ordinary differential equations},
  Prentice-Hall Mathematics Series, Prentice-Hall, Inc., Englewood Cliffs,
  N.J., 1961. \MR{0126573}

\bibitem[dC83]{doC1983}
Manfredo Perdig\~{a}o do~Carmo, \emph{Formas diferenciais e
  aplica\c{c}\~{o}es}, Monograf\'{\i}as de Matem\'{a}tica [Mathematical
  Monographs], vol.~37, Instituto de Matem\'{a}tica Pura e Aplicada (IMPA), Rio
  de Janeiro, 1983, With an appendix by Rubens Le\~{a}o. \MR{752287}

\bibitem[dC94]{doC1994}
Manfredo~P. do~Carmo, \emph{Differential forms and applications}, Universitext,
  Springer-Verlag, Berlin, 1994, Translated from the 1971 Portuguese original.
  \MR{1301070}

\bibitem[dC16]{doC2016}
\bysame, \emph{Differential geometry of curves \& surfaces}, Dover
  Publications, Inc., Mineola, NY, 2016, Revised \& updated second edition of
  [MR0394451]. \MR{3837152}

\bibitem[Hil01]{Hi1900}
David Hilbert, \emph{Ueber {F}l\"{a}chen von constanter {G}aussscher
  {K}r\"{u}mmung}, Trans. Amer. Math. Soc. \textbf{2} (1901), no.~1, 87--99.
  \MR{1500557}

\bibitem[Kui55]{K1955}
Nicolaas~H. Kuiper, \emph{On {$C^1$}-isometric imbeddings. {I}, {II}}, Nederl.
  Akad. Wetensch. Proc. Ser. A. {\bf 58} = Indag. Math. \textbf{17} (1955),
  545--556, 683--689. \MR{0075640}

\bibitem[Lee13]{Lee2013}
John~M. Lee, \emph{Introduction to smooth manifolds}, second ed., Graduate
  Texts in Mathematics, vol. 218, Springer, New York, 2013. \MR{2954043}

\bibitem[Lim03]{L2003}
Elon~Lages Lima, \emph{Fundamental groups and covering spaces}, A K Peters,
  Ltd., Natick, MA, 2003, Translated from the Portuguese by Jonas Gomes.
  \MR{2000701}

\bibitem[Mil72]{TKM1972}
Tilla~Klotz Milnor, \emph{Efimov's theorem about complete immersed surfaces of
  negative curvature}, Advances in Math. \textbf{8} (1972), 474--543.
  \MR{301679}

\bibitem[MP77]{MP1977}
Richard~S. Millman and George~D. Parker, \emph{Elements of differential
  geometry}, Prentice-Hall Inc., Englewood Cliffs, N. J., 1977. \MR{0442832}

\bibitem[Mun00]{M2000}
James~R. Munkres, \emph{Topology}, Prentice Hall, Inc., Upper Saddle River, NJ,
  2000, Second edition of [MR0464128]. \MR{3728284}

\bibitem[Nas54]{N1954}
John Nash, \emph{{$C^1$} isometric imbeddings}, Ann. of Math. (2) \textbf{60}
  (1954), 383--396. \MR{65993}

\bibitem[Nee97]{Ne1997}
Tristan Needham, \emph{Visual complex analysis}, The Clarendon Press, Oxford
  University Press, New York, 1997. \MR{1446490}

\bibitem[O'N06]{ON2006}
Barrett O'Neill, \emph{Elementary differential geometry}, second ed.,
  Elsevier/Academic Press, Amsterdam, 2006. \MR{2351345}

\bibitem[Rat94]{R1994}
John~G. Ratcliffe, \emph{Foundations of hyperbolic manifolds}, Graduate Texts
  in Mathematics, vol. 149, Springer-Verlag, New York, 1994. \MR{1299730}

\bibitem[Shi18]{Sh2018}
Theodore Shifrin, \emph{Differential geometry: a first course in curves and
  surfaces},
  \url{https://math.franklin.uga.edu/sites/default/files/inline-files/ShifrinDiffGeo.pdf},
  2018, preliminary version, accessed March 16, 2021.

\bibitem[Spi79a]{SI1979}
Michael Spivak, \emph{A comprehensive introduction to differential geometry.
  {V}ol. {I}}, second ed., Publish or Perish, Inc., Wilmington, Del., 1979.
  \MR{532830}

\bibitem[Spi79b]{SII1979}
\bysame, \emph{A comprehensive introduction to differential geometry. {V}ol.
  {II}}, second ed., Publish or Perish, Inc., Wilmington, Del., 1979.
  \MR{532831}

\bibitem[Wol77]{W1977}
Joseph~A. Wolf, \emph{Spaces of constant curvature}, fourth ed., Publish or
  Perish, Inc., Houston, TX, 1977. \MR{928600}

\end{thebibliography}

\end{document}